\documentclass[oneside,12pt]{amsart}
\usepackage{amscd,amsmath,amssymb,euscript}
\usepackage[frame,cmtip,curve,arrow,matrix,line,graph]{xy}

\usepackage{ mathrsfs }
\usepackage{ dsfont }

\usepackage{stmaryrd}
\usepackage{tikz}
\usetikzlibrary{positioning}
\tikzset{>=stealth}
\usetikzlibrary{arrows}
\usetikzlibrary{calc}
\usetikzlibrary{decorations.markings}
\usetikzlibrary{matrix}
\usetikzlibrary{calc}
\usepackage{tikz}
\usepackage{commath}
\usepackage[applemac]{inputenc}
\usepackage[T1]{fontenc}
\theoremstyle{plain}
\newtheorem{theorem}{Theorem}[section]
\newtheorem{lemma}{Lemma}

\newtheorem{conjecture}[theorem]{Conjecture}
\theoremstyle{definition}

\newtheorem{remark}[theorem]{Remark}

\usepackage{array}
\title{Local Pl\"ucker formulas for orthogonal groups}
\author{D. O. Degtyarev}

\address{AGHA Laboratory, Moscow Institute of Physics and Technology, Russia}
\email{degtyarev.d.o@gmail.com}
\begin{document}

\maketitle

\begin{abstract}
Local Pl\"ucker formulas relate the vector of metrics on a holomorphic curve in the projective space, induced by the Fubini-Study metrics on projetive spaces via Pl\"ucker embeddings of associated curves, to the corresponding vector of curvatures. Givental noted that the vector of curvatures is expressed in terms of the vector of metrics via Cartan matrix of type $ A_n$ and extended this observation from type $A_n$ to all nonexceptional types. In this paper we show how local Pl\"ucker formulas for special orthogonal groups, i.e. for Cartan matrices of type $ B_n$ and $ D_n$, can be obtained by reduction to the classical $A_n$ case.
\end{abstract}

\section{Introduction}

\subsection{Classical case}
Consider a holomorphic curve $S$ and a nondegenerate map $f =  f_1: S \to \mathbb{P}^n $. Define a local lift of $f_1$ in a neighborhood $U$ of any point $ p \in U \subset S $ by a holomorphic vector-valued function $ v(z) = (v_0 (z), \dots, v_n (z)): U\rightarrow \mathbb{C}^{n+1} $ and define the $ k $ -th ($ 1 \leqslant k \leqslant n $):
$$
f_k: U \to Gr(k,n+1)\subset \mathbb{P}\left(\bigwedge^{k}\mathbb{C}^{n+1}\right)
$$
by
$$
f_k(z) = [v(z)\wedge v'(z)\ldots v^{(k)}(z)].
$$

 Fubini-Study metric ($FS$) on projective space $ \mathbb{P}(\bigwedge^{k + 1} \mathbb{C}^{n + 1}) $ induces Hermitian metric on the $k$-th associated curve and the corresponding curvature forms via the Pl\"ucker embedding of $Gr(k,n+1)$. In both cases they are $(1,1)$ - forms on the associated curve. We will denote by $ \varphi_k$ the associated $(1,1)$-form of the metric on  $S$, and by $ \theta_k $ the curvature form corresponding to the metric $ \varphi_k $. The classical local Pl\"ucker formulas express the vector of curvatures $ \boldsymbol{\vartheta}: = (\vartheta_k) $ in terms of the vector of metrics $ \boldsymbol{\varphi}: = (\varphi_k) $ using the Cartan matrix $ C $ of the $ A_n $ series in the following form
\begin{equation}\label{classical plucker}
\boldsymbol{\vartheta} = C \boldsymbol{\phi}.
\end{equation}

\subsection{Local Pl\"ucker formulas according to Givental}
One can also construct another vector of associated $(1,1)$-forms of the metrics induced by pullback of the Fubini-Study metrics and the corresponding curvatures via the following morphisms of the associated curves to projective spaces.  We will follow the treatment from \cite{Giv}, \cite{Pos}.

Let $f:S \to F $ be a holomorphic curve on the flag variety $F: = G/B $ where $G$ is a complex semisimple Lie group of rank $n$, and $B$ is a Borel subgroup. We denote the tangent space to the point $ B \in F $ by $ T_{B} F $. Note that $T_{B} F$ is canonically isomorphic to $ \mathfrak{g}/ \mathfrak{b} $ where $ \mathfrak{g} $, $ \mathfrak{b} $ are Lie algebras of the groups $G$ and $B$ respectively. Consider the $n$-dimensional distribution $ \mathcal{N} $ on the flag variety $F$:
\begin{equation}
\mathcal{N}(B) = \{x \in \mathfrak{g} | [x,[\mathfrak{b},\mathfrak{b}]]\subset \mathfrak{b}\}/\mathfrak{b} \subset T_BF. 
\end{equation}

Let us now begin by considering the collection of line bundles $\mathscr{L}_1,\dots, \mathscr{L}_n $ on $F$ such that each $\mathscr{L}_i$ corresponds to the fundamental characters $\chi_i$ of the group $B$. The space $ \Gamma (F, \mathscr{L} _i)^{\vee} $ is the fundamental representation $V_i$ of the group $ G $. With the line bundle $ \mathscr{L}_i $ we can associate the projective morphism $ \pi_i $:
\begin{equation}
\pi_i: G/B \to \mathbb{P}(V_i)
\end{equation} 
$$
[g] \mapsto [g\cdot v_i], 
$$
where $v_i$ is highest weight vector of the representation  $V_i$ and $ g \in G $.

The representation $ V_i $ is irreducible, and therefore it has a Hermitian inner product invariant with respect to the action of the compact form $ K $ of the group $ G $. Note that a Hermitian inner product is uniquely defined up to multiplication
by a constant scalar factor. The Hermitian form determines the Fubini-Study metric $FS$ on the space $\mathbb{P}(V_i)$. The induced metric on the original curve $ S $ is denoted by $ \varphi_i = f^* \pi_i^*(FS) $ , and we write $ \theta_i $ for the corresponding curvature form. Givental conjectured in \cite{Giv}:

\begin{conjecture}\label{Giv theorem}
	If the holomorphic curve $ f: S \to F $ is an integral curve of the distribution $ \mathcal{N} $, then the vector $ \boldsymbol {\theta} $ of curvature forms is expressed in terms of the vector of metrics $ \boldsymbol{\varphi} $ using the Cartan matrix $ C $ of the Lie algebra $\mathfrak{g}$:
\end{conjecture}

\begin{equation}
\boldsymbol{\theta} = C \boldsymbol{\varphi}.
\end{equation}

\subsection{Connection with the classical case}

Consider the case when $ G = SL (n + 1, \mathbb{C}) $. The condition that the curve $ f: S \to F $ is an integral curve of the distribution $ \mathcal{N} $ means that the preimage under $ \pi_1: F \rightarrow \mathbb{P}(V)$ of the point $ x \in \pi_1 (S) $ is the flag of osculating $i$-planes of the curve $ \pi_1(S) $ at the point $ x $, i.e. the curve $ \pi_i (S) $ is the $i$-th associated of $ \pi_1 (S) $.

In this case the exterior power $ \bigwedge^i V$ of the standard representation of $G$ is the fundamental representation. The projective morphism $ \pi_i $ associated with the line bundle $ \mathscr{L} _i $ provides the Pl\"ucker embedding of the $i$-th associated curve on $Gr(i, n + 1) $ in $\mathbb{P}(V_i)$.

However Pl\"ucker embeddings of associated curves on Grassmannians do not always coincide with embeddings into the projectivization of fundamental representations. In particular this can be seen in the example of the special orthogonal groups.

Positselski proved the conjecture $ \ref{Giv theorem} $ in \cite{Pos} using the general technique for line bundles over flag varieties. At the same time in \cite{Giv},
in addition to the general formulation of the conjecture, Givental proposed the proof of the local Pl\"ucker formulas for symplectic group $ Sp(2n) $ based on the reduction of the problem to the classical formulas for the group $ SU(2n) $.

However, the question of the geometric interpretation in the case of a special orthogonal group has remained unanswered. In this paper we prove the $ \ref{Giv theorem} $ for orthogonal groups, i.e. for Cartan matrices of the series $ B_n $ and $ D_n $. Our proof is based on the reduction to the classical case.

\subsection*{Acknowledgements}
The reported study was funded by RFBR and CNRS, project number 21-51-15005. I am grateful to Leonid Positselski who introduced me to the problem some time ago.
\section{Local Pl\"ucker formulas for the group $SO (2n, \mathbb{C})$} 
Let us begin by considering the case when the group $ G = SO(2n, \mathbb{C}) $ with the Lie algebra $\mathfrak{g} = \mathfrak{so}(2n,\mathbb{B})$. Given a $2n$-dimensional vector space $V$ over $ \mathbb{C} $ equipped with a nondegenerate symmetric bilinear form $ Q $ which is preserved by $ G $. We will choose a basis $ e_1, \dots, e_n, f_1, \dots, f_{n} $ for $V$ in terms of which the quadratic form is given by
\begin{equation}\label{quadratic form}
Q (e_i, f_{i}) = Q(f_{i},e_i) = 1,
\end{equation}
$$ 
Q (e_i, e_j) = Q (f_i, f_j) = 0,
$$
and
$$
Q(e_s,f_k) =  Q(f_k,e_s) = 0,
$$
if $k \ne s$. The action of $G$ on the set of isotropic flags in $V$ with a given set of numerical invariants (i.e., the sequence of dimensions of its successive terms) is transitive except the case when maximal isotropic subspace in the flag has dimension $n$. In the latter case there are exactly two $G$-orbits with the given numerical invariants.

We will therefore consider two connected components of the variety of maximal isotropic flags in $V$. We denote these components by $F_+ $ and $ F_- $ and assume that the Borel subgroup $ B $ stabilizes the standard maximal isotropic flag $ \langle e_1 \rangle \subset \langle e_1, e_2 \rangle \subset ... \subset \langle e_1, ... e_n \rangle $ from $ F_+ $.

In contrast with the classical case the exterior powers of the tautological representation $ V_i: = \bigwedge^i V $ of $G$ are irreducible for $ i = 1, \ldots, n-1 $. Moreover, the exterior powers provide us the fundamental representations of the group $ G $ only for $ i = 1, \ldots, n-2 $. 

Let us denote by $ C (Q)^{even} $ the subalgebra of the Clifford algebra $C(Q)$ generated by elements of even degree with respect to a $\mathbb{Z}/2\mathbb{Z}$ grading. The remaining two fundamental representations $ V_i $ for $ i = n-1, n $ are realized as $C (Q)^{even}$-modules. We briefly recall the relevant constructions here. The decomposition $V = W \oplus W'$, where $W = \langle e_1,e_2,\ldots,e_n \rangle$ and $W' = \langle f_1,f_2,\ldots,f_n \rangle$, determines an isomorphism of algebras
\begin{equation}
C(Q) \cong \text{End}(\wedge^{\bullet} W),
\end{equation}
and $C(Q)^{even}$ respects the splitting $\bigwedge^{\bullet} W = \bigwedge^{even} W\oplus \bigwedge^{odd} W$. It then follows that 
\begin{equation}
C(Q)^{even} \cong \text{End}(\wedge^{even} W)\oplus \text{End}(\wedge^{odd} W).
\end{equation}
Moreover we have an embedding of Lie algebras
\begin{equation}
\mathfrak{so}(2n,\mathbb{C})\subset C(Q)^{even} \cong \mathfrak{gl}(\wedge^{even} W)\oplus \mathfrak{gl}(\wedge^{odd} W).
\end{equation}
These two representations $\wedge^{even} W$ and $\wedge^{odd} W$ are usually called half-spin representation (see \cite{FuHa}) .

Let us denote by $ E_{i,i} $ the diagonal matrix that  takes $ e_i $ to itself and kills $ e_j $ for $ j \ne i $. With the above-mentioned choice (\ref{quadratic form}), we may take as Cartan subalgebra $\mathfrak{h}$ the subalgebra of diagonal matrices generated by the $2n\times 2n$ matrices $H_i:=E_{i,i} - E_{i+n,i+n}$. We will correspondingly take as basis for the dual vector space $\mathfrak{h}^{*}$ the dual basis $L_i$ where $ L_i(H_j) = \delta_{i, j} $. For concreteness assume that the representation $ S_+ $ corresponds to the fundamental weight $\omega_{n-1} =  \omega_-: = \frac12 (L_1 + \ldots + L_ {n-1} -L_n) $ and the representation $ S_- $ corresponds to the fundamental weight $\omega_{n} =  \omega_+: = \frac12 (L_1 + \ldots + L_ {n-1} + L_n) $.

We also denote the morphisms $ \pi_{n-1} $ and $ \pi_n $ associated with the line bundles $ \mathscr{L}_{n-1} $ and $ \mathscr{L}_{n} $ by $ \pi_{-} $ and $\pi_+$ respectively, i.e. suppose that $ \pi_{\pm}: G / B \to \mathbb{P} (S_{\pm}) $, $ \pi_{\pm}: g \mapsto [g v_{\pm}] $ , where $ v_{+} $ and $v_{-}$ are highest weight vectors of the representations $ S_{+} $ and $S_-$.

Let $ i_{+}: G / B \to Fl(V) $ be the embedding that corresponds to the orbit of the standard maximal isotropic flag, i.e. $ i_+: [b] \mapsto \langle e_1 \rangle \subset \langle e_1, e_2 \rangle \subset \ldots \subset \langle e_1, e_2, \ldots e_{n-1}, e_n \rangle $. Define the analogous embedding that corresponds to the another orbit $ i_{-}: G / B \to Fl(V) $, $[b] \mapsto \langle e_1 \rangle \subset \langle e_1, e_2 \rangle \subset \ldots \subset \langle e_1, e_2, \ldots e_{n-1}, f_n \rangle $. Denote the composition of $i_{\pm}$, projection to $ Gr (n, 2n) $ and the Pl\"ucker embedding by $p_{n-1}$:
\begin{equation}
 p_{n-1}: G / B \rightarrow Fl(V) \rightarrow Gr(n-1,2n) \rightarrow \mathbb{P}(\wedge^{n-1} V)    
\end{equation}
For the future reference we prove the following lemma.

\begin{lemma}
Let $ \sigma: \mathbb{P}(S_ +) \times \mathbb{P}(S_-) \to \mathbb{P} (S_+ \otimes S_-) $ be the Segre embedding. Then there is a natural linear embedding $ e: \mathbb{P} (\wedge^{n-1} V) \to \mathbb{P} (S_+ \otimes S_-) $ such that the following diagram is commutative:
	\begin{equation}\label{segre diagram}
	\begin{tikzpicture}
	\matrix (m) [matrix of math nodes,row sep=3em,column sep=4em,minimum width=2em]
	{
		\, & \mathbb{P}(S_+) \times \mathbb{P}(S_-) \\
		G/B & \mathbb{P}(S_+ \otimes S_-) \\
		\, & \mathbb{P}(\wedge^{n-1}V)\\
	};
	\path[-stealth]
	(m-2-1) edge node [left] {$(\pi_+,\pi_-)\,\,\,$} (m-1-2)
	(m-2-1) edge node [left] {$p_{n-1}\,\,$} (m-3-2)
	(m-1-2) edge node [right]{$\sigma$} (m-2-2)
	(m-3-2) edge node [right] {$e$} (m-2-2);
	\end{tikzpicture}
	\end{equation}
	
\end{lemma}

\begin{proof}
	Note that $ \wedge^{n-1} V $ is the irreducible representation of $\mathfrak{g}$  with the highest weight $ \omega_+ + \omega_{-} $. The complete decomposition into irreducible representations of $ S_+ \otimes S_- $ looks like
	\begin{equation}\label{tensor decomposition of spin}
	S_+\otimes S_- = \wedge^{n-1}V \bigoplus\limits_{j>0}V_{n - 2j-1},
	\end{equation}
	where $V_i$ is the irreducible representation of $\mathfrak{g}$ with the highest weight $\omega_i = L_1+L_2+\ldots+L_{n - 2j-1}$.
	Note that the vector $ v _ + \otimes v _- $ has weight $ \omega_n + \omega_ {n-1} $ and therefore it is the highest weight vector of the first summand in (\ref{tensor decomposition of spin}). This completes the proof of the lemma.
	
\end{proof}

The projective space $ \mathbb{P}(S_+ \otimes S_-) $ admits a $ K $ -invariant Fubini - Study metric induced by the $ K $ -invariant Hermitian inner products on the spaces $ S_+ $ and $ S_- $. We denote the corresponding $ (1,1) $ - form of the metric on the space $ \mathbb {P} (S_+ \otimes S_-) $ by $ FS $.
The following standard properties of $(1,1)$-forms of the metrics will be needed later.
\begin{lemma}
	Let $ N, M $ be complex manifolds, $ f: N \to M $ a holomorphic map such that the induced map $ f_{*}: T'_z (N) \to T '_ {f (z)} (M) $ is injective for all $z \in N$, then the associated (1,1) -form of the induced metric on $ N $ is the pullback of the associated (1,1) -form of the metric on $ M $.
\end{lemma}

\begin{lemma}\label{metric segre}
	Let $ \mathbb{P}^n $ and $ \mathbb{P}^m $ be projective spaces with homogeneous coordinates $ [z_0, \ldots, z_n] $ and $ [w_0, \ldots, w_m] $ respectively, equipped with the standard Fubini-Study metrics, $ \sigma $ is the Segre embedding $ \sigma: \mathbb{P}^n \times \mathbb{P}^m \to \mathbb{P}^N $
	\begin{equation}
	\sigma([z_0,\ldots,z_n], [w_0,\ldots, w_m] = [(z_iw_j)_{0\leqslant i \leqslant n, 0\leqslant j \leqslant m}].
	\end{equation}  
	Denote by $ \alpha: \mathbb{P}^n \times \mathbb {P}^m \to \mathbb{P}^n $ and $ \beta: \mathbb{P}^n \times \mathbb{P}^m \to \mathbb{P}^m $ the projections onto the factors, then $ \sigma $ is a holomorphic isometry, that is
	\begin{equation}
	\sigma^*FS = \alpha^*FS+ \beta^* FS
	\end{equation}
\end{lemma}
 We shall need the following fact.
\begin{lemma}\label{n-1 metric}
	Let $ \varphi_+ $ and $ \varphi_- $ denote the pullback of the $ FS $ metrics on the spaces $ \mathbb{P}(S_+) $ and $ \mathbb{P} (S_-) $ via morphims $\pi_+\circ f$ and $\pi_-\circ f$ respectively and let $\phi_ {n-1} $ be the metric on the curve $ S $ induced by pullback of the metric $FS$ on the space $ \mathbb{P} (\wedge^{n-1} V) $ via $p_{n-1}\circ f$. Then the following equality holds 
	\begin{equation} \label{segre metric sum}
	\varphi_ + + \varphi_- = \phi_{n-1}. \end{equation}
\end{lemma}
\begin{proof}
	The commutative diagram (\ref{segre diagram}) shows that on the one hand
	\begin{equation}\label{plucker metric}
	(e\circ p_{n-1} \circ f)^*FS = (p_{n-1}\circ f)^*FS = \phi_{n-1}.
	\end{equation}
	On the other hand, it follows from the above lemmas that the following equalities hold up
	\begin{equation}\label{segre metric}
	(\sigma \circ (\pi_+,\pi_-)\circ f)^* FS = ((\pi_+,\pi_-)\circ f)^* (\alpha^*FS+\beta^*FS) =
	\end{equation}
	$$
	=  (\pi_+ \circ f)^*FS+(\pi_-\circ f)^*FS = \varphi_+ +\varphi_-.
	$$
	To prove the lemma it is sufficient to note that the left-hand sides of (\ref{plucker metric}) and (\ref{segre metric}) coincide. Therefore the lemma is proven.
\end{proof}
Consider the morphisms $p_{\pm}$:
\begin{equation}
p_{\pm}: G / B \to Fl(V) \to Gr(n, 2n)  \subset \mathbb{P}(\wedge^n \mathbb{C}^{2n})  
\end{equation} 
which are compositions of the embeddings $ i_{\pm}$, projections to $ Gr(n, 2n) $ and the Pl\"ucker embedding.
We shall need the following properties of the morphisms $p_{\pm}$.
\begin{lemma}\label{veronese_lemma}
	Let $ v_{\pm}: \mathbb{P}(S_{\pm}) \to \mathbb{P} (Sym^2S_{\pm}) $ be the quadratic Veronese embeddings, then there exist natural linear embeddings $ e_{\pm}: \mathbb{P} (\wedge^nV_{\pm}) \to  \mathbb{P} (\wedge^n Sym^2S _{\pm}) $ such that the following diagram is commutative
	\begin{equation}\label{spin diagram2}
	\begin{tikzpicture}
	\matrix (m) [matrix of math nodes,row sep=3em,column sep=4em,minimum width=2em]
	{
		\, & \mathbb{P}(S_{\pm}) \\
		G/B & \mathbb{P}(\mathrm{Sym}^2S_{\pm})\\
		\, & \mathbb{P}(\wedge^{n}V_{\pm})\\
	};
	\path[-stealth]
	(m-2-1) edge node [left] {$\pi_{\pm}\,$} (m-1-2)
	(m-1-2) edge node [right] {$v_{\pm}\,\,$} (m-2-2)
	(m-3-2) edge node [left] {$\,\,$} (m-2-2)
	(m-3-2) edge node [right]{$e_{\pm}$} (m-2-2)
	(m-2-1) edge node [left]{$p_{\pm}$} (m-3-2);
	\end{tikzpicture}
	\end{equation}
	
\end{lemma}

\begin{proof}
	The representation $ \wedge^{n} V $ is no longer an irreducible representation of $\mathfrak{g}$ as in the classical case. It decomposes into a direct sum of the irreducible representations $ \wedge^{n} V_{-} $ and $ \wedge^{n} V_+ $ with highest weights $ 2 \omega_{-} $ and $ 2 \omega_{+} $ respectively.

	Note also that there is the following direct sum decomposition of $ Sym^2S_{\pm} $ into irreducible representations of $\mathfrak{g}$:
	\begin{equation}
	\mathrm{Sym}^2 S_{\pm} = \wedge^{n}V_{\pm}\oplus \bigoplus\limits_{j>0}\wedge^{n-4j}V
	\end{equation}
	
	Since the vectors $ v_+ \otimes v_+ $ and $ v_- \otimes v_- $ are vectors of highest weight $ 2 \omega_+ $ and $ 2 \omega_{-} $ respectively, we get the natural linear embeddings $e_{\pm}$.
\end{proof}

\begin{remark} The morphisms $ \pi_{+} $ and $\pi_{-}$ decompose into a composition as follows

\begin{equation}
\pi_{\pm}: G / B \to G / P_{\pm} \xrightarrow {spin} \mathbb{P} (S_{\pm}),
\end{equation}
Note that $ \mathrm{Pic} (G/P_{\pm}) \simeq \mathbb{Z} $ due to the fact that $ P_{\pm} $ is a maximal parabolic subgroup, and the positive generator of the group corresponds to the very ample bundle $ \mathcal{O}_{\mathbb {S}_{\pm}} (1) $ where $\mathbb{S}_{+}$ and $\mathbb{S}_{+}$ are spinor varieties. The embeddings $spin_{\pm}$ are associated to these generators.
\end{remark}

\begin{lemma}\label{metric veronese}
	
	Let $ \phi_+ $ and $\phi_- $ denote the pullback of $ FS $ metrics from the spaces $\mathbb{P}(\wedge^nV_+) $ and $\mathbb{P}(\wedge^nV_-) $ via morphims $p_+\circ f$ and $p_-\circ f$ respectively and let $\phi_ {n-1} $ be the metric on the curve $ S $ induced by pullback of metrics $FS$ from the spaces $ \mathbb{P} (\wedge^{n-1} V_+) $ and $ \mathbb{P} (\wedge^{n-1} V_-) $ via morphims $p_{+}\circ f$ and $p_{-}\circ f$. Then the following equalities hold $\phi_+ = 2 \varphi_+ $ and $ \phi_- = 2 \varphi_- $.
\end{lemma}
\begin{proof}
	 In order to prove the above statement we need to return to the commutative diagram (\ref{spin diagram2}). The spaces $\mathbb{P}(Sym^2S_{+})$ and $\mathbb{P}(Sym^2S_{-}) $ have the Fubini-Study metric induced by the Hermitian scalar product on the spaces $S_+\otimes S_+$ and $S_- \otimes S_-$ respectively.
	
	Note that, the Segre embedding $\sigma_ +: \mathbb{P}(S_+) \to \mathbb{P}(S_+ \otimes S_+)$ factors through $\mathbb{P}(S_+) \times \mathbb{P}(S_+)$:
	
	\begin{equation}
v_+: \mathbb{P}(S_+) \xrightarrow{\Delta}\mathbb{P}(S_+) \times \mathbb{P}(S_+)\xrightarrow{\sigma} \mathbb{P}(S_+ \otimes S_+),
	\end{equation}
	where $\Delta$ is the diagonal morphism, $\sigma$ -- Segre embedding. We already know by lemma \ref{metric segre} that  $ v_+^*(FS) = 2 FS $. The similar relation holds for the morphism $ v_- $. Thus we get:

	\begin{equation}
	2 \varphi_{\pm} = (v_{\pm} \circ \pi_{\pm} \circ f)^* FS = (e_{\pm}\circ p_{\pm}  \circ f)^* FS = \phi_{ \pm}
	\end{equation} 
	
\end{proof}

Let $ \theta(\mathscr{L}) $ denote the curvature of the $ K $ - invariant metric in an arbitrary bundle $ \mathscr{L} $ with the action of $ G $ on $ G / B $. Let $l$ be a local section of $ \mathscr{L} $, then the following formula holds:
\begin{equation}
\theta(\mathscr{L}) = -\partial \overline{\partial}\mathrm{ln}\|l\|.
\end{equation}

We slightly abuse notation and take certainly a standard candidate for Cartan subalgebra of $\mathfrak{sl}(2n,\mathbb{C})$, namely the subalgebra of diagonal matrices with zero trace and write $L_i$ for the element of the dual basis such that $L_i(E_{j,j}) = \delta_{i,j}$. 
\begin{lemma} Suppose $ \mathcal{N}_{\pm} $ are line bundles on the projective variety $ G / B $ corresponding to the simple roots $ \lambda_{\pm} = L_{n-1} \pm L_{n} $ of $ \mathfrak{so}(2n, \mathbb{C}) $, and $ \mathcal{N}_{n-1} $ is the line bundle on the full flag varietiy $Fl(V)$ corresponding to the simple root $ \lambda_{n-1 } = L_{n-1} -L_n $ of $\mathfrak{sl}(2n,\mathbb{C})$. Then there are following identities:
	$$
	i_{\pm}^{*}(\theta(\mathscr{N}_{n-1}) = \theta(\mathscr{N}_{\mp})
	$$
\end{lemma}

Classical Pl\"ucker formulas for the group $SL(2n,\mathbb{C})$ give the following relation on the curve $S$:
\begin{equation} \label{classic plucker}
 \boldsymbol{\vartheta} = C \boldsymbol{\phi},
\end{equation}
where $C$ is a Cartan matrix of type $A_{2n-1} $. We notice that relation (\ref{classic plucker}) descends to the analogous relation on $Fl(V)$:
\begin{equation}\label{claccical plucker on grassmannian}
 \widehat{\boldsymbol{\vartheta}} = C \widehat{\boldsymbol{\phi}} 
\end{equation}
where $\widehat{\boldsymbol{\vartheta}}$ is the vector of $(1,1)$-forms $\widehat{\vartheta}_i:= \theta(\mathcal{N}_i)$ on $Fl(V)$, $\mathcal{N}_i$ is the line bundle on $Fl(V)$ corresponding to the simple root $\lambda_i = L_i-L_{i+1}$ of $\mathfrak{sl}(2n,\mathbb{C})$, $\widehat{\boldsymbol{\phi}}$ -- vector of $(1,1)$-forms with $\widehat{\phi}_i:=\widehat{p}_i^*(FS)$ and $\widehat{p}_i: Fl(V) \rightarrow \mathbb{P}(\wedge^i V) $ is the standard projection.

\begin{theorem}
	The local Pl\"ucker formulas for the group $SO(2n, \mathbb{C})$ are obtained by taking pullback of the classical Pl\"ucker formulas on $Fl(V)$ via the maps $i_{\pm}\circ f$ of the first $n-1$ components of the vector $ \widehat{\boldsymbol{\vartheta}} $.
\end{theorem}
\begin{proof}
 We need only deal with the first $n-1$ components of the vector $\widehat{\boldsymbol{\vartheta}}$. The first $ n-3 $ components of the vector  $\widehat{\boldsymbol{\vartheta}}$ remain unchanged as the pullback is taken with respect to morphisms $i_{\pm}$. Note that the result does not depend on the choice of embedding $i_+$ or $ i_-$.

Consider the $ (n-2) $-th row of the matrix relation (\ref{claccical plucker on grassmannian}):
	\begin{equation}
	\widehat{\vartheta}_{n-2} = -\widehat{\phi}_{n-3}+2\widehat{\phi}_{n-2}-\widehat{\phi}_{n-1}.
	\end{equation}
	Taking into the account formulas ($\ref{segre metric sum}$) and applying pullback to the right- and left-hand sides we have
	\begin{equation}\label{n-2 curvature on grassmannian}
	\theta_{n-2} = (i_+\circ f)^*\widehat{\vartheta}_{n-2} = -(i_+\circ f)^*\widehat{\phi}_{n-3}+ (i_+\circ f)^*2\widehat{\phi}_{n-2}-(i_+\circ f)^*\widehat{\phi}_{n-1}= 
	\end{equation}
	$$
	= -\varphi_{n-3}+2 \varphi_{n-2}-\varphi_+-\varphi_-.
	$$
	Next we consider the $(n-1)$-th component of the curvature vector $ \widehat{\boldsymbol{\vartheta}}$ and the corresponding row of the matrix equality ($\ref{claccical plucker on grassmannian}$):
	\begin{equation}\label{n-1 curvature on grassmannian}
	\widehat{\vartheta}_{n-1} = -\widehat{\phi}_{n-2}+2\widehat{\phi}_{n-1}-\widehat{\phi}_{n}.
	\end{equation}
	In view of Lemma \ref{n-1 metric} and Lemma \ref{metric veronese} we observe that pullback of the right- and left-hand sides of the equality (\ref{n-1 curvature on grassmannian}) via the morphism $i_+\circ f$ yields the desiring equality for the one of the last two rows of Pl\"ucker relations for $SO(2n,\mathbb{C})$:
	
	\begin{equation}
	\theta_- = (i_{+}\circ f)^*(\widehat{\vartheta}_{n-1}) = -(i_{+}\circ f)^*(\widehat{\phi}_{n-2})+ (i_{+}\circ f)^*(2\widehat{\phi}_{n-1})-(i_+\circ f)^*(\widehat{\phi}_{n}) = 
	\end{equation}
	$$
	= -\varphi_{n-2} +2 \varphi_+ +2\varphi_- - 2\varphi_+= -\varphi_{n-2}+2\varphi_{-}.
	$$
	Then the same reasoning as before shows that the remaining row of Pl\"ucker formulas for $SO(2n,\mathbb{C})$ is given by:
	\begin{equation}
	\theta_+ = (i\circ f)^*_{-}(\widehat{\vartheta}_{n-1}) = -(i\circ f)^*_{-}(\widehat{\phi}_{n-2})+ (i\circ f)^*_{-}(2\widehat{\phi}_{n-1})-(i\circ f)^*_-(\widehat{\phi}_{n}) = 
	\end{equation}
	$$
	= -\omega_{n-2} +2 \omega_+ +2\omega_- - 2\omega_-= -\omega_{n-2}+2\omega_{+}.
	$$
	With matrix notation we get the full collection of Pl\"ucker formulas on the curve $S$ for $SO(2n,\mathbb{C})$:
	\begin{equation}
	\boldsymbol{\theta} = A \boldsymbol{\varphi},
	\end{equation}
	where $A$ is a Cartan matrix of type $D_n$.
\end{proof}     

\section{Local Pl\"ucker formulas for the group $ SO (2n + 1, \mathbb{C})$} 
Consider the case $G = SO(2n + 1, \mathbb{C}) $. Let $V$ be a complex vector space of dimension $2n + 1$ with a nondegenerate symmetric bilinear form $Q$ which is preserved by $G$. In this case the exterior power $V_i: = \wedge^i V $ of the standard representation $V$ of $\mathfrak{so}(2n+1,\mathbb{C})$ is the irreducible representation with highest weight $L_1 + \ldots + L_i$ for $i = 1, \dots n$. Exterior powers of standard representation provide us the fundamental representations of the group $G$ only for $i = 1,\dots, n-1$. The remaining fundamental representations $V_n$ with highest weight $\omega_n:=(L_1 + \ldots + L_n)/2$ is called the spin representation and is realized on $S:=\wedge^{\bullet} W$ where $W \subset V $ is the $n$ -dimensional isotropic subspace. As in the case of even orthogonal Lie algebras, we can consider the projective morphism associated with line bundle $ \mathscr{L}_{i}$ on $G/B$ corresponding to the fundamental weight $\omega_i$.

We immediately get the first $(n-2)$-th rows of Pl\"ucker formulas for $SO(2n+1,\mathbb{C})$. We briefly show that two remaining formulas for $ i = n-1, n $ can be reconstructed from classical formulas.

Note that for the odd case a lemma similar to the lemma $ \ref{veronese_lemma} $ is true.
\begin{lemma}\label{veronese_lemma_odd}
	Let $ v: \mathbb{P}(V_n) \to \mathbb{P} (V_n \otimes V_n)) $ be the quadratic Veronese embedding and let  $p_n: G/B \rightarrow \mathbb{P}(\wedge^nV) $ be the the composition of the standard projection to Grassmannian with Pl\"ucker embedding, then there exists a natural linear embedding $ e: \mathbb{P} (\wedge^n V) \to \mathbb{P}(V_n \otimes V_n) $ such that the following diagram is commutative
	\begin{equation}\label{spin diagram}
	\begin{tikzpicture}
	\matrix (m) [matrix of math nodes,row sep=3em,column sep=4em,minimum width=2em]
	{
		\, & \mathbb{P}(V_n) \\
		G/B & \mathbb{P}(V_n\otimes V_n)\\
		\, & \mathbb{P}(\wedge^{n}V)\\
	};
	\path[-stealth]
	(m-2-1) edge node [left] {$\pi_n\,$} (m-1-2)
	(m-1-2) edge node [right] {$v\,\,$} (m-2-2)
	(m-3-2) edge node [right] {$e\,$} (m-2-2)
	(m-2-1) edge node [left]{$p_n$} (m-3-2);
	\end{tikzpicture}
	\end{equation}
	
\end{lemma}
\begin{proof}
	The representation $ V_n \otimes V_n $ of $\mathfrak{so}(2n+1,\mathbb{C})$ decomposes into a direct sum
	\begin{equation}
	V_n \otimes V_n = \bigwedge^{n} V \oplus \bigwedge^{n-1} V \oplus \ldots \oplus \bigwedge^0 V.
	\end{equation}
	Let $ v_n $ be the highest weight vector with weight $\omega_n$ of the representation $ V_n $. Next observe that $ v_n \otimes v_n $ is the vector with weight $ L_1 + \ldots + L_n $, i.e. $v_n\otimes v_n$ is the highest weight vector of the irreducible representation  $ \bigwedge^n V $. The lemma is an immediate consequence of this fact.
\end{proof}
The technical statements in the odd dimensional case can be worked out as for an even orthogonal groups.
\begin{lemma}	Let $ \varphi_n $ denote the pullback of the $ FS $ metric on the space $ \mathbb{P}(S)$ via morphims $\pi_n\circ f$ and let $\phi_ {n-1} $ be the metric on the curve $ S $ induced by pullback of the metric $FS$ on the space $ \mathbb{P} (\wedge^{n} V) $ via $p_{n}\circ f$. Then the following equality holds 
	\begin{equation}
	2\varphi_{n} = \phi_{n}. \end{equation}
\end{lemma}

Combining the previous two lemmas, we can now formulate the orthogonal version of Pl\"ucker formulas for odd orthogonal groups. 

\begin{theorem}
The local Pl\"ucker formulas for the group $SO(2n+1, \mathbb{C})$ are obtained by taking pullback of the classical Pl\"ucker formulas on $Fl(V)$ via the map $i\circ f$ of the first $n$ components of the vector $ \widehat{\boldsymbol{\vartheta}} $.
\end{theorem}

\end{document}